\documentclass[reqno]{amsart}
\usepackage{amsfonts}
\usepackage[reqno]{amsmath}
\usepackage{amscd}
\usepackage{amssymb}
\usepackage{graphicx}

\providecommand{\U}[1]{\protect \rule{.1in}{.1in}}
\newtheorem{theorem}{Theorem}

\newtheorem{lemma}[theorem]{Lemma}

\theoremstyle{definition}
\newtheorem{definition}[theorem]{Definition}

\theoremstyle{remark}
\newtheorem{remark}{Remark}

\input{tcilatex}

\begin{document}
\title[Finite biorthogonal polynomials in one variable]{Finite biorthogonal
polynomials suggested by the finite orthogonal polynomials $M_{n}^{\left(
p,q\right) }\left( x\right) $}
\author{Esra G\"{U}LDO\u{G}AN LEKES\.{I}Z}
\address{Department of Mathematics, Faculty of Science, Gazi University,
Ankara 06500, T\"{u}rkiye}
\email{esragldgn@gmail.com}

\begin{abstract}
In this paper, we derive a pair of finite univariate biorthogonal
polynomials suggested by the finite univariate orthogonal polynomials $%
M_{n}^{\left( p,q\right) }\left( x\right) $. The corresponding
biorthogonality relation is given. Some useful relations and properties,
concluding differential equation and generating function, are presented.
Further, a new family of finite biorthogonal functions is obtained using
Fourier transform and Parseval identity. In addition, we compute the Laplace
transform and fractional calculus operators for polynomials $M_{n}\left(
p,q,\upsilon ;x\right) $.
\end{abstract}

\maketitle

\section{Introduction}

Biorthogonal polynomials dates back to Didon \cite{Didon} and Deruyts \cite%
{Deruyts}. There are extensive usage in several braches of mathematics and
also physics such as quantum mechanics and electrostatics. In recent years,
the biorthogonal polynomial theory has been studied and expanded by many
researchers \cite%
{AlSalamVerma,Carlitz2,Carlitz3,Chai,Dou,ITV,Konhauser,Konhauser2,MT0,MT2,MT,Rassias,Seker,KT,P,PK}%
. Further, as is known, biorthogonal polynomials reduce to orthogonal
polynomials in a special case.

The main ones of well-known biorthogonal polynomials in one variable are
Konhauser polynomials \cite{Konhauser2} suggested by the Laguerre
polynomials and the biorthogonal polynomials, which are reduced to Jacobi
polynomials in a special case \cite{MT}. Studies started with definition of
Laguerre matrix polynomials being extentions of the Laguerre polynomials 
\cite{JCN}. Then they have reached today with a similar extension of
classical orthogonal polynomials \cite{DJL,DJ,JCP,JCP2}.

\begin{definition}
Assume that polynomials $d(x)$ and $f(x)$ are of degree $k$ and $l$ for $k>0$
and$\ l>0$. $D_{r}(x)$ and$\ F_{n}(x)$ represent polynomials of degree $r$
and $n$ with respect to (w.r.t.) real-valued polynomials $d(x)$ and $f(x)$,
respectively. Then polynomials $D_{r}(x)$ and$\ F_{n}(x)$ are of degree $rk$
and $nl$, respectively. Thus, polynomials $d(x)$ and $f(x)$ are said to be
fundamental polynomials \cite{Konhauser}.
\end{definition}

\begin{definition}
Let $\rho \left( x\right) $ be a real-valued weight function on an interval $%
(a,b)$ if%
\begin{equation}
I_{j,s}=\int\limits_{a}^{b}\rho \left( x\right) \left[ d\left( x\right) %
\right] ^{j}\left[ f\left( x\right) \right] ^{s}dx,\ \ \ j,s=0,1,2,...
\label{d}
\end{equation}%
holds, with%
\begin{equation*}
I_{0,0}=\int\limits_{a}^{b}\rho \left( x\right) dx\neq 0.
\end{equation*}%
If the integrals (\ref{d}) exist for $j,s=0,1,2,...$, then%
\begin{equation*}
\int\limits_{a}^{b}\rho \left( x\right) x^{s}dx,\ \ \ s=0,1,2,...
\end{equation*}%
holds \cite{Konhauser}.
\end{definition}

\begin{definition}
For $r,n\in 
\mathbb{N}
_{0}$, if%
\begin{equation*}
J_{r,n}=\int\limits_{a}^{b}\rho \left( x\right) D_{r}(x)F_{n}(x)dx=\QATOPD\{
. {\ \ \ \ \ 0;\ \ r\neq n}{\text{not }0;\ \ r=n},
\end{equation*}%
then the sets of the polynomials $\{D_{r}(x)\}$ and $\{F_{n}(x)\},$
corresponding to the fundamental polynomials $d(x)$ and $f(x)$, are called
biorthogonal w.r.t. the weight function $\rho (x)$ on $(a,b)$ \cite%
{Konhauser}.
\end{definition}

Equivalent to this definition, the following theorem can be given for
defining biorthogonal polynomials.

\begin{theorem}
Assume that $\rho (x)$ is a weight function on $(a,b)$. For $r,n\in 
\mathbb{N}
_{0}$, if the fundamental polynomials $d\left( x\right) $ and $f\left(
x\right) $ are such that%
\begin{equation}
\int\limits_{a}^{b}\rho \left( x\right) \left[ d\left( x\right) \right]
^{j}F_{n}(x)dx=\QATOPD\{ . {\ \ 0,\ j\neq n}{\text{not }0,\ j=n\ \ \ \ },
\label{a}
\end{equation}%
and%
\begin{equation}
\int\limits_{a}^{b}\rho \left( x\right) \left[ f\left( x\right) \right]
^{j}D_{r}(x)dx=\QATOPD\{ . {\ \ 0,\ j\neq r}{\text{not }0,\ j=r\ \ \ \ },
\label{b}
\end{equation}%
then%
\begin{equation}
\int\limits_{a}^{b}\rho \left( x\right) D_{r}(x)F_{n}(x)dx=\QATOPD\{ . {\ \
\ \ \ 0,\ \ r\neq n}{\text{not }0,\ \ r=n}  \label{c}
\end{equation}%
is satisfied and, conversely, when (\ref{c}) holds then both (\ref{a}) and (%
\ref{b}) hold \cite{Konhauser}.
\end{theorem}

Biorthogonal polynomials emerge as $n$-th degree polynomial solutions of the
differential equation%
\begin{equation*}
A_{1}(x)y^{\prime \prime \prime }+A_{2}(x)y^{\prime \prime
}+A_{3}(x)y^{\prime }=\lambda y
\end{equation*}%
w.r.t. $x^{n}$ and the $n$-th degree polynomial solutions of the adjoint
equation of this equation w.r.t. $x$ \cite{Preiser}. Biorthogonal
polynomials were first described by Fano and Spencer in 1951 \cite{SF}. For
being $\gamma $ is a non-negative integer, they just investigated the
biorthogonality of polynomials in terms of $x$ and $x^{2}$, but not describe
any general properties of the biorthogonal polynomials. They used these
polynomials to calculate the Gamma rays entering matter. In 1965, the most
general properties of the biorthogonal polynomials, such as zeros and the
existence of recurrence relations, were studied by Konhauser \cite{Konhauser}%
. Konhauser defined the self-titled polynomials $Z_{n}^{\left( \gamma
\right) }(x;\upsilon )$ and $Y_{n}^{\left( \gamma \right) }(x;\upsilon )$ in
1967, which are biorthogonal w.r.t. the weight function $x^{\gamma }e^{-x}$
on the interval $\left( 0,\infty \right) $ \cite{Konhauser2}. Al-Salam and
Verma gave the q expansion of these polynomials in 1983 \cite{AlSalamVerma}.

The concept of biorthogonal polynomials was developed by several
mathematicians in the later years. In particular, the families of
biorthogonal polynomials suggested by the classical orthogonal polynomials
have been studied and many properties of these families, such as recurrence
relations and Rodrigues formulas, have been obtained \cite{MT0,MT2}. Spencer
and Fano used a special version of Konhauser polynomials in the calculation
of Gamma rays penetrating objects. In fact, q analogs of these polynomials
have been made and used also in approximation theory \cite{ITV,Seker,Seker2}.

For $\gamma >-1$ and $\upsilon =1,2,...$, the Konhauser polynomials $%
Z_{n}^{\left( \gamma \right) }(x;\upsilon )$ and $Y_{n}^{\left( \gamma
\right) }(x;\upsilon )$ are the following pair of biorthogonal polynomials
w.r.t. the weight function $w\left( x\right) =x^{\gamma }e^{-x}$ on $\left(
0,\infty \right) $ that are suggested by the Laguerre polynomials:%
\begin{equation}
Z_{n}^{\left( \gamma \right) }\left( x;\upsilon \right) =\frac{\Gamma \left(
\upsilon n+\gamma +1\right) }{n!}\sum\limits_{j=0}^{n}\left( -1\right) ^{j}%
\binom{n}{j}\frac{x^{\upsilon j}}{\Gamma \left( \upsilon j+\gamma +1\right) }
\label{5}
\end{equation}%
and%
\begin{equation}
Y_{n}^{\left( \gamma \right) }\left( x;\upsilon \right) =\frac{1}{n!}%
\sum\limits_{r=0}^{n}\frac{x^{r}}{r!}\sum\limits_{j=0}^{r}\left( -1\right)
^{j}\binom{r}{j}\left( \frac{1+\gamma +j}{\upsilon }\right) _{n},  \label{6}
\end{equation}%
where $\upsilon =1,2,3,...$. These two sets have the biorthogonality
relation \cite{Konhauser2}%
\begin{equation}
\int\limits_{0}^{\infty }e^{-x}x^{\gamma }Z_{n}^{\left( \gamma \right)
}\left( x;\upsilon \right) Y_{r}^{\left( \gamma \right) }(x;\upsilon
)dx=\QATOPD\{ . {\frac{\Gamma \left( \upsilon n+\gamma +1\right) }{n!}%
,r=n}{0\ \ \ \ \ \ \ \ \ \ \ \ ,r\neq n}.  \label{Konhauserbiort}
\end{equation}%
Taking account of (\ref{a}), (\ref{b}) and (\ref{Konhauserbiort}), the
following theorem can be given:

\begin{theorem}
Polynomials $Z_{n}^{\left( \gamma \right) }\left( x;\upsilon \right) $ and$\
Y_{r}^{\left( \gamma \right) }(x;\upsilon )$ defined in (\ref{5}) and (\ref%
{6}) satisfy the following biorthogonality relations w.r.t. the weight
function $e^{-x}x^{\gamma }$ on $\left( 0,\infty \right) $, respectively:%
\begin{equation*}
\int\limits_{0}^{\infty }e^{-x}x^{\gamma }Z_{n}^{\left( \gamma \right)
}\left( x;\upsilon \right) x^{j}dx=\QATOPD\{ . {\text{not }0,\ j=n\ \ \ \ \
\ }{0,\ j\neq n}
\end{equation*}%
and%
\begin{equation*}
\int\limits_{0}^{\infty }e^{-x}x^{\gamma }Y_{r}^{\left( \gamma \right)
}(x;\upsilon )x^{\upsilon j}dx=\QATOPD\{ . {\text{not }0,\ j=r\ \ \ \ \ \
}{0,\ j\neq r}.
\end{equation*}
\end{theorem}

\bigskip

On the other hand, for $p,q>-1$ and $\upsilon =1,2,...$, the pair of
biorthogonal polynomials $J_{n}\left( p,q,\upsilon ;x\right) $ and $%
K_{n}\left( p,q,\upsilon ;x\right) $ suggested by the Jacobi polynomials, is
defined by%
\begin{equation*}
J_{n}\left( p,q,\upsilon ;x\right) =\frac{\left( 1+p\right) _{\upsilon n}}{n!%
}\sum\limits_{j=0}^{n}\left( -1\right) ^{j}\binom{n}{j}\frac{\left(
1+p+q+n\right) _{\upsilon j}}{\left( 1+p\right) _{\upsilon j}}\left( \frac{%
1-x}{2}\right) ^{\upsilon j}
\end{equation*}%
and%
\begin{equation*}
K_{n}\left( p,q,\upsilon ;x\right)
=\sum\limits_{r=0}^{n}\sum\limits_{s=0}^{r}\frac{\left( -1\right)
^{r+s}\left( 1+q\right) _{n}}{n!r!\left( 1+q\right) _{n-r}}\binom{r}{s}%
\left( \frac{s+p+1}{\upsilon }\right) _{n}\left( \frac{x-1}{2}\right)
^{r}\left( \frac{x+1}{2}\right) ^{n-r}.
\end{equation*}%
The corresponding biorthogonality relation is of form%
\begin{eqnarray*}
&&\int\limits_{-1}^{1}\left( 1-x\right) ^{p}\left( 1+x\right)
^{q}J_{n}\left( p,q,\upsilon ;x\right) K_{r}\left( p,q,\upsilon ;x\right) dx
\\
&=&\frac{2^{p+q+1}\Gamma \left( p+\upsilon n+1\right) \Gamma \left(
q+n+1\right) }{n!\Gamma \left( p+q+n+1\right) \left( p+q+\upsilon
n+n+1\right) }\delta _{n,r}.
\end{eqnarray*}%
For $\upsilon =1$, both polynomials $J_{n}\left( p,q,\upsilon ;x\right) $
and $K_{n}\left( p,q,\upsilon ;x\right) $ reduce to the classical Jacobi
polynomials and are called as biorthogonal polynomials suggested by the
Jacobi polynomials.

Also, the polynomials $J_{n}\left( p,q,\upsilon ;x\right) $ satisfy the
following recurrence relations:%
\begin{equation}
\left( x-1\right) DJ_{n}\left( p,q,\upsilon ;x\right) =\upsilon nJ_{n}\left(
p,q,\upsilon ;x\right) -\upsilon \left( \upsilon n-\upsilon +p+1\right)
_{\upsilon }\ J_{n-1}\left( p,q+1,\upsilon ;x\right)  \label{Jrec1}
\end{equation}%
and%
\begin{equation}
\left( x-1\right) DJ_{n}\left( p,q,\upsilon ;x\right) =\left( \upsilon
n+p\right) J_{n}\left( p-1,q+1,\upsilon ;x\right) -p\ J_{n}\left(
p,q,\upsilon ;x\right) .  \label{Jrec2}
\end{equation}

\bigskip

The idea of constructing biorthogonal pairs of polynomials corresponding to
the weight functions of classical orthogonal polynomials was found out by
Konhauser \cite{Konhauser2} for the Laguerre weight function $x^{\gamma
}e^{-x}$, by Toscano \cite{Toscano}, Carlitz \cite{Carlitz3} and
Madhekar\&Thakare \cite{MT} for the Jacobi weight function $(1-x)^{\alpha
}(1+x)^{\beta }$ and by Madhekar\&Thakare \cite{MT2} for the Hermite weight
function $\exp \left( -x^{2}\right) $. For $\mu >-1/2$, the Szeg\"{o}%
-Hermite polynomials $H_{n}^{\left( \mu \right) }(x)$ are biorthogonal
w.r.t. Szeg\"{o}-Hermite weight function $\left\vert x\right\vert ^{2\mu
}exp(-x^{2})$ over the interval $(-\infty ,\infty )$ and these have been
experienced to be useful in connection with Gauss-Jacobi mechanical
quadrature \cite{S}. For $\mu =0$, Szeg\"{o}-Hermite polynomials are just
the classical Hermite polynomials.

\bigskip

On the other hand, finite orthogonal polynomials in one variable and two
variables have been studied in many papers in the literature \cite%
{GAM,GAM2,GA,GL4,KM,Masjed}. In theory of orthogonal polynomials, the term
"finite" is used when the degree of the polynomial is limited depending on
one or more parameters of the polynomial.

One of well-known finite orthogonal polynomials in one variable is the
polynomials $M_{n}^{\left( p,q\right) }\left( x\right) $, which is one of
the three finite solutions of the equation%
\begin{equation*}
x\left( x+1\right) y_{n}^{\prime \prime }\left( x\right) +\left( \left(
2-p\right) x+1+q\right) y_{n}^{\prime }\left( x\right) -n\left( n+1-p\right)
y_{n}\left( x\right) =0.
\end{equation*}%
These polynomials, defined by%
\begin{equation*}
M_{n}^{\left( p,q\right) }\left( x\right) =\left( -1\right)
^{n}n!\sum_{j=0}^{n}\binom{p-n-1}{j}\binom{q+n}{n-j}\left( -x\right) ^{j},
\end{equation*}%
satisfy the orthogonality relation%
\begin{equation*}
\int\limits_{0}^{\infty }x^{q}\left( 1+x\right) ^{-\left( p+q\right)
}M_{r}^{\left( p,q\right) }\left( x\right) M_{n}^{\left( p,q\right) }\left(
x\right) dx=\frac{n!\Gamma \left( p-n\right) \Gamma \left( q+n+1\right) }{%
\left( p-2n-1\right) \Gamma \left( p+q-n\right) }\delta _{r,n}
\end{equation*}%
under the conditions that $p>2\left\{ \max n\right\} +1$ and$\ q>-1$\ \cite%
{Masjed}. $\delta _{r,n}$ is the Kronecker delta.

\bigskip

In this study, inspired by papers \cite{Konhauser2} and \cite{MT}, we define
a pair of finite biorthogonal polynomials, namely $M_{n}\left( p,q,\upsilon
;x\right) $ and $\mathfrak{M}_{n}\left( p,q,\upsilon ;x\right) $, suggested
by finite orthogonal polynomials $M_{n}^{\left( p,q\right) }\left( x\right) $
and obtain the corresponding biorthogonality relation. By this means, the
notion of finitude is transferred to the theory of biorthogonal polynomials
for the first time in this paper. Also, it is shown that the new families
are related to the biorthogonal polynomials suggested by the Jacobi
polynomials. Moreover, a limit relation between this finite set and the
Konhauser polynomials is introduced. Several properties such as recurrence
relation, generating function and differential equation are presented. Next,
computing its Fourier transform, a new family of finite biorthogonal
functions is derived via Parseval's identity. In addition to integral and
operational representation, Laplace transform and Riemann-Liouville
fractional calculus operators of $M_{n}\left( p,q,\upsilon ;x\right) $ are
considered.

\section{A pair of finite biorthogonal polynomials}

In this paper we introduce a pair of biorthogonal polynomials suggested by
the finite univariate orthogonal polynomials $M_{n}^{\left( p,q\right)
}\left( x\right) $.

For $p>\left( \upsilon +1\right) N+1$, $q>-1$, $N=\max \left\{ n\right\} $, $%
n=0,1,2,...$ and $\upsilon =0,1,2,...$, $M_{n}\left( p,q,\upsilon ;x\right) $
and $\mathfrak{M}_{n}\left( p,q,\upsilon ;x\right) $ are the polynomials of
degree $n$ in $x^{\upsilon }$ and $x$, respectively. These two polynomial
sets satisfy the following biorthogonality condition w.r.t. the weight
function $x^{q}\left( 1+x\right) ^{-\left( p+q\right) }$, namely%
\begin{equation}
\int\limits_{0}^{\infty }x^{q}\left( 1+x\right) ^{-\left( p+q\right)
}M_{n}\left( p,q,\upsilon ;x\right) x^{j}dx=\left\{ \QATOP{0\ \text{for}\
j\neq n}{\text{not}\ 0\text{ for}\ j=n\ \ \ \ \ \ }\right. ,  \label{1}
\end{equation}%
and%
\begin{equation}
\int\limits_{0}^{\infty }x^{q}\left( 1+x\right) ^{-\left( p+q\right) }%
\mathfrak{M}_{n}\left( p,q,\upsilon ;x\right) x^{\upsilon j}dx=\left\{ 
\QATOP{0\ \text{for}\ j\neq n}{\text{not}\ 0\text{ for}\ j=n\ \ \ \ \ \ }%
\right. .  \label{2}
\end{equation}%
It follows from (\ref{1}) and (\ref{2}) that%
\begin{equation*}
\int\limits_{0}^{\infty }x^{q}\left( 1+x\right) ^{-\left( p+q\right)
}M_{n}\left( p,q,\upsilon ;x\right) \mathfrak{M}_{r}\left( p,q,\upsilon
;x\right) dx=\left\{ \QATOP{0\ \text{for}\ r\neq n}{\text{not}\ 0\text{ for}%
\ r=n\ \ \ \ \ \ }\right.
\end{equation*}%
and conversely, for $p>\left( \upsilon +1\right) N+1,\ q>-1,N=\max \left\{
n,r\right\} .$

\begin{remark}
For $\upsilon =1$ both these sets reduce to the sets of the finite
univariate orthogonal polynomails $M_{n}^{\left( p,q\right) }\left( x\right) 
$.\bigskip
\end{remark}

Madhekar and Thakare \cite{MT} introduced the biorthogonal polynomials
suggested by the Jacobi polynomials with the help of the following
connection between the classical Jacobi polynomials $P_{n}^{\left(
p,q\right) }\left( x\right) $ and the Laguerre polynomials $L_{n}^{\left(
p\right) }\left( x\right) $:%
\begin{equation}
\Gamma \left( n+p+q+1\right) P_{n}^{\left( p,q\right) }\left( x\right)
=\int\limits_{0}^{\infty }t^{n+p+q}e^{-t}L_{n}^{\left( p\right) }\left( 
\frac{1-x}{2}t\right) dt.  \label{7}
\end{equation}%
Similarly, we consider the relation (\ref{7}) to generate the finite
biorthogonal pair.\bigskip

On the other hand, let us recall the transition between the classical
orthogonal Jacobi polynomials and the finite univariate orthogonal
polynomials $M_{n}^{\left( p,q\right) }\left( x\right) $ of form \cite%
{Masjed}:%
\begin{eqnarray}
M_{n}^{\left( p,q\right) }\left( x\right) =\left( -1\right)
^{n}n!P_{n}^{\left( q,-p-q\right) }\left( 2x+1\right) &&  \label{7.} \\
\Leftrightarrow P_{n}^{\left( p,q\right) }\left( x\right) =\frac{\left(
-1\right) ^{n}}{n!}M_{n}^{\left( -p-q,p\right) }\left( \frac{x-1}{2}\right)
. &&  \notag
\end{eqnarray}%
By considering (\ref{7}), the results in (\ref{7.}) enable us to introduce
the first set from the pair of biorthogonal polynomials, called $M_{n}\left(
p,q,\upsilon ;x\right) $ and $\mathfrak{M}_{n}\left( p,q,\upsilon ;x\right) $%
, suggested by the finite univariate orthogonal polynomials $M_{n}^{\left(
p,q\right) }\left( x\right) $.

Let us define the first set $M_{n}\left( p,q,\upsilon ;x\right) $ by%
\begin{equation}
\frac{\left( -1\right) ^{n}}{n!}\Gamma \left( n+p+q+1\right) M_{n}\left(
-p-q,p,\upsilon ;x\right) =\int\limits_{0}^{\infty
}t^{n+p+q}e^{-t}Z_{n}^{\left( p\right) }\left( -xt;\upsilon \right) dt.
\label{8}
\end{equation}%
From (\ref{5}) and (\ref{8}), we can easily construct%
\begin{equation}
M_{n}\left( p,q,\upsilon ;x\right) =\left( -1\right) ^{n}\left( q+1\right)
_{\upsilon n}\sum\limits_{j=0}^{n}\left( -1\right) ^{j}\binom{n}{j}\frac{%
\left( n+1-p\right) _{\upsilon j}}{\left( q+1\right) _{\upsilon j}}\left(
-x\right) ^{\upsilon j}.  \label{9}
\end{equation}%
In fact, $M_{n}\left( p,q,\upsilon ;x\right) $ can be expressed in terms of
the generalized hypergeometric functions of form%
\begin{equation*}
M_{n}\left( p,q,\upsilon ;x\right) =\left( -1\right) ^{n}\left( q+1\right)
_{\upsilon n}\ _{\upsilon +1}F_{\upsilon }\left[ \QATOP{-n,\ \Delta \left(
\upsilon ,n+1-p\right) }{\Delta \left( \upsilon ,q+1\right) }\left(
-x\right) ^{\upsilon }\right] ,
\end{equation*}%
where $\Delta \left( \upsilon ,\gamma \right) $ states the sequence of $%
\upsilon $ parameters $\frac{\gamma }{\upsilon },\frac{\gamma +1}{\upsilon }%
,...,\frac{\gamma +\upsilon -1}{\upsilon },\ \upsilon \geq 1.$

Now, we think the following explicit serial representation of form%
\begin{equation}
\mathfrak{M}_{n}\left( p,q,\upsilon ;x\right)
=\sum\limits_{r=0}^{n}\sum\limits_{s=0}^{r}\left( -1\right) ^{s+n}\binom{r}{s%
}\frac{\left( p+q-n\right) _{r}}{r!}\left( \frac{s+q+1}{\upsilon }\right)
_{n}x^{r}\left( 1+x\right) ^{n-r},  \label{11}
\end{equation}%
as the second set of the pair of the finite biorthogonal polynomials
suggested by the finite orthogonal polynomials $M_{n}^{\left( p,q\right)
}\left( x\right) $.

\begin{remark}
Both $M_{n}\left( p,q,\upsilon ;x\right) $ and $\mathfrak{M}_{n}\left(
p,q,\upsilon ;x\right) $ get reduced for $\upsilon =1$ to the finite
univariate orthogonal polynomials $M_{n}^{\left( p,q\right) }\left( x\right) 
$.
\end{remark}

\section{Biorthogonality}

\begin{theorem}
The pair of polynomials $M_{n}\left( p,q,\upsilon ;x\right) $ and $\mathfrak{%
M}_{m}\left( p,q,\upsilon ;x\right) $ given by (\ref{9}) and (\ref{11}) are
finite biorthogonal and satisfies the finite biorthogonality relation%
\begin{equation}
\int\limits_{0}^{\infty }\frac{x^{q}}{\left( 1+x\right) ^{p+q}}M_{n}\left(
p,q,\upsilon ;x\right) \mathfrak{M}_{m}\left( p,q,\upsilon ;x\right) dx=%
\frac{n!\Gamma \left( p-n\right) \Gamma \left( q+1+\upsilon n\right) \delta
_{n,m}}{\left( p-1-n-\upsilon n\right) \Gamma \left( p+q-n\right) },
\label{Mort}
\end{equation}%
for $p>\left( \upsilon +1\right) N+1,\ N=\max \left\{ n,m\right\} $ and $%
q>-1 $.
\end{theorem}

\begin{proof}
\begin{eqnarray}
&&\int\limits_{0}^{\infty }x^{q}\left( 1+x\right) ^{-\left( p+q\right)
}M_{n}\left( p,q,\upsilon ;x\right) \mathfrak{M}_{m}\left( p,q,\upsilon
;x\right) dx  \label{biortrel} \\
&=&\left( -1\right) ^{n}\left( q+1\right) _{\upsilon
n}\sum\limits_{j=0}^{n}\left( -1\right) ^{j+\upsilon j}\binom{n}{j}\frac{%
\left( n+1-p\right) _{\upsilon j}}{\left( q+1\right) _{\upsilon j}}  \notag
\\
&&\times \sum\limits_{r=0}^{m}\sum\limits_{s=0}^{r}\left( -1\right) ^{s+m}%
\binom{r}{s}\frac{\left( p+q-m\right) _{r}}{r!}\left( \frac{s+q+1}{\upsilon }%
\right) _{m}\int\limits_{0}^{\infty }x^{q+\upsilon j+r}\left( 1+x\right)
^{m-r-p-q}dx  \notag \\
&=&\frac{\left( -1\right) ^{n+m}\Gamma \left( q+1+\upsilon n\right) \Gamma
\left( p-1-m\right) }{\Gamma \left( p+q-m\right) }\sum\limits_{j=0}^{n}%
\left( -1\right) ^{j}\binom{n}{j}\frac{\left( n+1-p\right) _{\upsilon j}}{%
\left( m+2-p\right) _{\upsilon j}}  \notag \\
&&\times \sum\limits_{r=0}^{m}\binom{q+\upsilon j+r}{r}\sum\limits_{s=0}^{r}%
\left( -1\right) ^{s}\binom{r}{s}\left( \frac{s+q+1}{\upsilon }\right) _{m}.
\notag
\end{eqnarray}%
By using the following result given by Carlitz \cite[p. 429]{Carlitz2}%
\begin{equation*}
\left( \frac{x+\alpha +1}{\upsilon }\right) _{n}=\sum\limits_{r=0}^{n}\binom{%
-x+r-1}{r}\sum\limits_{s=0}^{r}\left( -1\right) ^{s}\binom{r}{s}\left( \frac{%
s+\alpha +1}{\upsilon }\right) _{n},
\end{equation*}%
(\ref{biortrel}) becomes%
\begin{eqnarray*}
&&\int\limits_{0}^{\infty }x^{q}\left( 1+x\right) ^{-\left( p+q\right)
}M_{n}\left( p,q,\upsilon ;x\right) \mathfrak{M}_{m}\left( p,q,\upsilon
;x\right) dx \\
&=&\frac{\left( -1\right) ^{n+m}\Gamma \left( q+1+\upsilon n\right) \Gamma
\left( p-1-m\right) }{\Gamma \left( p+q-m\right) }\sum\limits_{j=0}^{n}%
\left( -1\right) ^{j}\binom{n}{j}\left( -j\right) _{m}\frac{\left(
n+1-p\right) _{\upsilon j}}{\left( m+2-p\right) _{\upsilon j}} \\
&=&\left( -1\right) ^{n+m+1}\frac{\Gamma \left( q+1+\upsilon n\right) \Gamma
\left( p-m\right) m!}{\Gamma \left( p+q-m\right) }\binom{n}{m}%
\sum\limits_{j=0}^{n-m}\left( -1\right) ^{j}\binom{n-m}{j}\frac{\left(
n+1-p\right) _{\upsilon \left( j+m\right) }}{\left( m+1-p\right) _{\upsilon
\left( j+m\right) }} \\
&=&\frac{m!\Gamma \left( p-m\right) \Gamma \left( q+1+\upsilon n\right)
\left( 1-p\right) _{m}}{\left( -1\right) ^{n+m+1}\Gamma \left( p+q-m\right)
\left( 1-p\right) _{n}}\binom{n}{m}\sum\limits_{j=0}^{n-m}\left( -1\right)
^{j}\binom{n-m}{j}\frac{\left( 1-p\right) _{n+\upsilon \left( j+m\right) }}{%
\left( 1-p\right) _{m+\upsilon \left( j+m\right) +1}} \\
&=&\frac{m!\Gamma \left( p-m\right) \Gamma \left( q+1+\upsilon n\right)
\left( 1-p\right) _{m}}{\left( -1\right) ^{n+m+1}\Gamma \left( p+q-m\right)
\left( 1-p\right) _{n}}\binom{n}{m}\sum\limits_{j=0}^{n-m}\left( -1\right)
^{j}\binom{n-m}{j}D^{n-m-1}x^{n-p+\upsilon \left( j+m\right) }\mid _{x=1} \\
&=&\QATOPD\{ . {\ \ \ 0,\ \ m\neq n}{\neq 0,\ \ m=n}.
\end{eqnarray*}
\end{proof}

\section{Some useful relations}

\begin{lemma}
The following limit relations hold between the Konhauser polynomials and the
finite biorthogonal polynomials:%
\begin{equation}
\left\{ \QATOP{\lim\limits_{p\rightarrow \infty }M_{n}\left( p,q,\upsilon ;%
\frac{x}{p}\right) =\left( -1\right) ^{n}n!Z_{n}^{\left( q\right) }\left(
x;\upsilon \right) ,}{\lim\limits_{p\rightarrow \infty }\mathfrak{M}%
_{n}\left( p,q,\upsilon ;\frac{x}{p}\right) =\left( -1\right)
^{n}n!Y_{n}^{\left( q\right) }\left( x;\upsilon \right) .}\right.  \label{12}
\end{equation}
\end{lemma}

\begin{remark}
For $\upsilon =1$, each of (\ref{12}) reduces to the well known limit
relations in \cite[p.17]{Masjed2}, between the Laguerre polynomials and the
finite univariate orthogonal polynomials $M_{n}^{\left( p,q\right) }\left(
x\right) $.
\end{remark}

\begin{lemma}
We have the following transitions between the pair of biorthogonal
polynomials given by (\ref{9}) and (\ref{11}), and the set of the
biorthogonal polynomials suggested by the Jacobi polynomials:%
\begin{eqnarray}
M_{n}\left( p,q,\upsilon ;x\right) =\left( -1\right) ^{n}n!J_{n}\left(
q,-p-q,\upsilon ;2x+1\right) &&  \label{JM} \\
\Leftrightarrow J_{n}\left( p,q,\upsilon ;x\right) =\frac{\left( -1\right)
^{n}}{n!}M_{n}\left( q,-p-q,\upsilon ;\frac{x-1}{2}\right) &&  \notag
\end{eqnarray}%
and%
\begin{eqnarray}
\mathfrak{M}_{n}\left( p,q,\upsilon ;x\right) =\left( -1\right)
^{n}n!K_{n}\left( q,-p-q,\upsilon ;2x+1\right) &&  \label{KM} \\
\Leftrightarrow K_{n}\left( p,q,\upsilon ;x\right) =\frac{\left( -1\right)
^{n}}{n!}\mathfrak{M}_{n}\left( q,-p-q,\upsilon ;\frac{x-1}{2}\right) . && 
\notag
\end{eqnarray}
\end{lemma}

\section{Generating functions and recurrence relations}

\begin{theorem}
Polynomials $M_{n}\left( p,q,\upsilon ;x\right) $ have the following
generating functions%
\begin{eqnarray}
&&\sum\limits_{n=0}^{\infty }\frac{\left( 1-q\right) _{n}}{\left(
1-p-q\right) _{\upsilon n}}M_{n}\left( p,q,\upsilon ;x\right) \frac{t^{n}}{n!%
}  \label{13} \\
&=&\left( 1+t\right) ^{q-1}\ _{\upsilon +1}F_{\upsilon }\left[ \QATOP{\Delta
\left( \upsilon +1,1-q\right) ;}{\Delta \left( \upsilon ,1-p-q\right) ;}%
\left( -\frac{\left( \upsilon +1\right) x}{\upsilon }\right) ^{\upsilon }%
\frac{\left( \upsilon +1\right) t}{\left( 1+t\right) ^{\upsilon +1}}\right] 
\notag
\end{eqnarray}%
and%
\begin{equation}
\sum\limits_{n=0}^{\infty }\frac{1}{\left( 1-p-q\right) _{\upsilon n}}%
M_{n}\left( p-n,q+n,\upsilon ;x\right) \frac{t^{n}}{n!}=e^{-t}\ _{\upsilon
}F_{\upsilon }\left[ \QATOP{\Delta \left( \upsilon ,1-q\right) ;}{\Delta
\left( \upsilon ,1-p-q\right) ;}\left( -x\right) ^{\upsilon }t\right] .
\label{14}
\end{equation}
\end{theorem}

\begin{proof}
From the definition (\ref{9}) we obtain (\ref{13}) and (\ref{14}) by usual
series techniques.
\end{proof}

\begin{theorem}
The recurrence relations%
\begin{equation}
DM_{n}\left( p,q,\upsilon ;x\right) =-\upsilon n\left( -x\right) ^{\upsilon
-1}\left( n+1-p\right) _{\upsilon }\ M_{n-1}\left( 1-p-q,p-1-\upsilon
,\upsilon ;x\right) ,  \label{15}
\end{equation}%
\begin{eqnarray}
xDM_{n}\left( p,q,\upsilon ;x\right) &=&\upsilon n\ M_{n}\left(
-p-q,p,\upsilon ;x\right)  \label{16} \\
&&+\upsilon n\left( \upsilon n-\upsilon +q+1\right) _{\upsilon }\
M_{n-1}\left( 1-p-q,p-1,\upsilon ;x\right)  \notag
\end{eqnarray}%
and%
\begin{equation}
xDM_{n}\left( p,q,\upsilon ;x\right) =\left( \upsilon n+q\right) M_{n}\left(
1-p-q,p,\upsilon ;x\right) -q\ M_{n}\left( -p-q,p,\upsilon ;x\right)
\label{17}
\end{equation}%
are satisfied for the polynomials $M_{n}\left( p,q,\upsilon ;x\right) $.
\end{theorem}

\begin{proof}
If we differentiate (\ref{9}) w.r.t. $x$, we get the relation (\ref{15}).

(\ref{16}) and (\ref{17}) are proved by using relation (\ref{JM}) and the
recurrence relation (\ref{Jrec1}) and (\ref{Jrec2}), for biorthogonal
polynomials suggested by the Jacobi polynomials, respectively.
\end{proof}

\begin{theorem}
Polynomials $M_{n}\left( p,q,\upsilon ;x\right) $ satisfy the following
differential equation%
\begin{equation}
\left[ xD\left( xD+q+1-\upsilon \right) _{\upsilon }-\left( -x\right)
^{\upsilon }\left( xD-\upsilon n\right) \left( xD+n+1-p\right) _{\upsilon }%
\right] M_{n}\left( p,q,\upsilon ;x\right) =0.  \label{Mdifequ}
\end{equation}
\end{theorem}

\begin{proof}
$M_{n}\left( p,q,\upsilon ;x\right) $ are essentially $_{\upsilon
+1}F_{\upsilon }$-type generalized hypergeometric polynomials, and the
generalized hypergeometric function $_{p}F_{q}$ satisfies the equation \cite%
{R}%
\begin{equation*}
v\left( v+\beta _{1}-1\right) \left( v+\beta _{2}-1\right) ...\left( v+\beta
_{q}-1\right) F\left( x\right) =x\left( v+\alpha _{1}\right) \left( v+\alpha
_{2}\right) ...\left( v+\alpha _{p}\right) F\left( x\right) ,
\end{equation*}%
where $v=x\frac{\partial }{\partial x}$ is the differential operator. So, we
have the differential equation (\ref{Mdifequ}).
\end{proof}

\begin{theorem}
The generating function%
\begin{equation}
\sum\limits_{n=0}^{\infty }\mathfrak{M}_{n}\left( p-n,q+n,\upsilon ;x\right) 
\frac{t^{n}}{2^{n}\left( 1+x\right) ^{n}n!}=\left( 1+x\right) ^{-p}\left( 1+%
\frac{1}{2t}\right) ^{\frac{q-1}{\upsilon }}\left[ x+\left( 1+\frac{1}{2t}%
\right) ^{1/\upsilon }\right] ^{p}  \label{Mdog}
\end{equation}%
holds for $\mathfrak{M}_{n}\left( p,q,\upsilon ;x\right) $.
\end{theorem}

\begin{proof}
From the definition (\ref{11}), we get (\ref{Mdog}).
\end{proof}

\begin{theorem}
The following formula is satisfied for $\mathfrak{M}_{n}\left( p,q,\upsilon
;x\right) $:%
\begin{equation}
\mathfrak{M}_{n}\left( p-n,q+n,\upsilon ;x\right) =\left( -1\right)
^{n}\left( 1+x\right) ^{n-p}\left\{ \frac{\partial ^{n}}{\partial t^{n}}%
\left( 1-t\right) ^{\frac{q-1}{\upsilon }}\left[ x+\left( 1-t\right)
^{1/\upsilon }\right] ^{p}\right\} \mid _{t=0}.  \label{18}
\end{equation}
\end{theorem}

\begin{proof}
(\ref{18}) follows from (\ref{Mdog}).
\end{proof}

\begin{theorem}
Polynomials $\mathfrak{M}_{n}\left( p,q,\upsilon ;x\right) $ have the
following generating function:%
\begin{eqnarray}
\sum\limits_{n=0}^{\infty }\mathfrak{M}_{n}\left( p+\zeta n,q+\theta
n,\upsilon ;x\right) \frac{u^{n}}{n!}=\upsilon \left( 1+t\right)
^{-p-q}\left( 1+\frac{xt}{1+x}\right) ^{p}\ \ \ \ \ \ \ \ \ \ \ \ \ \ \ \ \
\ \ \ \ \ \ \ \ &&  \label{Mgen} \\
\times \left\{ \frac{\upsilon +\left( \zeta +\theta \right) \left[ \left(
1+t\right) ^{\upsilon }-1\right] }{1+t}-\frac{x\left( 1+\zeta \right) \left[
\left( 1+t\right) ^{\upsilon }-1\right] }{1+x+xt}\right\} ^{-1}, &&  \notag
\end{eqnarray}%
where $u=-t\left( 1+x\right) ^{\zeta }\left( 1-t\right) ^{\frac{1-\theta }{%
\upsilon }}\left[ x+\left( 1-t\right) ^{1/\upsilon }\right] ^{-1-\zeta }$.
\end{theorem}

\begin{proof}
By considering the generating function \cite[p. 146]{PS} and the relation (%
\ref{KM}), we arrive the generating function (\ref{Mgen}).
\end{proof}

\section{Fourier transform of the finite univariate biorthogonal polynomials
suggested by the finite orthogonal polynomials $M_{n}^{\left( p,q\right)
}\left( x\right) $}

The Fourier transform for a function $s(x)$ in one variable is defined as 
\cite[p. 111, Equ. (7.1)]{Davies}%
\begin{equation*}
\tciFourier \left( s\left( x\right) \right) =\int\limits_{-\infty }^{\infty
}e^{-i\varpi x}s\left( x\right) dx
\end{equation*}%
and the corresponding Parseval identity is given by the statement%
\begin{equation}
\int\limits_{-\infty }^{\infty }s\left( x\right) \overline{r\left( x\right) }%
dx=\frac{1}{2\pi }\int\limits_{-\infty }^{\infty }\tciFourier \left( s\left(
x\right) \right) \overline{\tciFourier \left( r\left( x\right) \right) }%
d\varpi  \label{Parseval}
\end{equation}%
for $s,r\in L^{2}\left( 
\mathbb{R}
\right) $.

\bigskip

If we define the specific functions%
\begin{equation*}
\QATOPD\{ . {s\left( x\right) =e^{\gamma _{1}x}\left( 1+e^{x}\right)
^{-\gamma _{1}-\gamma _{2}}M_{n}\left( p,q,\upsilon ;e^{x}\right) ,}{r\left(
x\right) =e^{\lambda _{2}x}\left( 1+e^{x}\right) ^{-\lambda _{1}-\lambda
_{2}}\mathfrak{M}_{m}\left( a,b,\upsilon ;e^{x}\right) ,}
\end{equation*}%
then%
\begin{eqnarray*}
&&\tciFourier \left( s\left( x\right) \right) =\int\limits_{-\infty
}^{\infty }e^{-i\varpi x}s\left( x\right) dx=\int\limits_{-\infty }^{\infty
}e^{\left( \gamma _{1}-i\varpi \right) x}\left( 1+e^{x}\right) ^{-\left(
\gamma _{1}+\gamma _{2}\right) }M_{n}\left( p,q,\upsilon ;e^{x}\right) dx \\
&=&\int\limits_{-\infty }^{\infty }u^{\gamma _{1}-i\varpi }\left( 1+u\right)
^{-\left( \gamma _{1}+\gamma _{2}\right) }M_{n}\left( p,q,\upsilon
;e^{x}\right) dx \\
&=&\left( -1\right) ^{n}\left( q+1\right) _{\upsilon n}\sum_{j=0}^{n}\frac{%
\left( -1\right) ^{\upsilon j}\left( -n\right) _{j}\left( n+1-p\right)
_{\upsilon j}}{j!\left( q+1\right) _{\upsilon j}}\int\limits_{0}^{\infty
}u^{\gamma _{1}-i\varpi -1+\upsilon j}\left( 1+u\right) ^{-\left( \gamma
_{1}+\gamma _{2}\right) }du \\
&=&\left( -1\right) ^{n}B\left( \gamma _{1}-i\varpi ,\gamma _{2}+i\varpi
\right) \left( q+1\right) _{\upsilon n}\sum_{j=0}^{n}\frac{\left( -n\right)
_{j}\left( n+1-p\right) _{\upsilon j}\left( \gamma _{1}-i\varpi \right)
_{\upsilon j}}{j!\left( q+1\right) _{\upsilon j}\left( 1-\gamma _{2}-i\varpi
\right) _{\upsilon j}} \\
&=&B\left( \gamma _{1}-i\varpi ,\gamma _{2}+i\varpi \right) \left( -1\right)
^{n}\left( q+1\right) _{\upsilon n} \\
&&\times \ _{\upsilon +1}F_{2\upsilon }\left[ -n,\Delta \left( \upsilon
,n+1-p\right) ;\Delta \left( \upsilon ,q+1\right) ,\Delta \left( \upsilon
,1-\gamma _{2}-i\varpi \right) ;1\right] ,
\end{eqnarray*}%
where%
\begin{equation*}
B(p,q)=\int\limits_{0}^{\infty }x^{p-1}\left( 1-x\right) ^{q-1}dx=\frac{%
\Gamma \left( p\right) \Gamma \left( q\right) }{\Gamma \left( p+q\right) }%
=B\left( q,p\right)
\end{equation*}%
and%
\begin{equation*}
\Gamma \left( z\right) =\int\limits_{0}^{\infty }t^{z-1}e^{-t}dt\ \ ,\ \ 
\func{Re}(z)>0
\end{equation*}%
are the Beta and the Gamma integrals, respectively.

Similarly,%
\begin{eqnarray*}
\tciFourier \left( r\left( x\right) \right) &=&\frac{\Gamma \left( \lambda
_{1}+i\varpi \right) \Gamma \left( \lambda _{2}-i\varpi \right) }{\Gamma
\left( \lambda _{1}+\lambda _{2}-m\right) \left( 1-\lambda _{1}-i\varpi
\right) _{m}} \\
&&\times \sum_{r=0}^{m}\frac{\left( a+b-m\right) _{r}\left( \lambda
_{2}-i\varpi \right) _{r}}{\left( \lambda _{1}+\lambda _{2}-m\right) _{r}r!}%
\sum_{s=0}^{r}\frac{\left( -r\right) _{s}}{s!}\left( \frac{s+b+1}{\upsilon }%
\right) _{m}.
\end{eqnarray*}

Thus, from Parseval identity (\ref{Parseval}), we write%
\begin{eqnarray}
&&\int\limits_{0}^{\infty }u^{\gamma _{1}+\lambda _{2}-1}\left( 1+u\right)
^{-\left( \gamma _{1}+\gamma _{2}+\lambda _{1}+\lambda _{2}\right)
}M_{n}\left( p,q,\upsilon ;u\right) \mathfrak{M}_{m}\left( a,b,\upsilon
;u\right) du  \label{P1} \\
&=&\frac{\left( -1\right) ^{n}\left( q+1\right) _{\upsilon n}}{2\pi i\Gamma
\left( \gamma _{1}+\gamma _{2}\right) \Gamma \left( \lambda _{1}+\lambda
_{2}-m\right) }\int\limits_{-\infty }^{\infty }\Gamma \left( \gamma
_{1}-i\varpi \right) \Gamma \left( \gamma _{2}+i\varpi \right) \Gamma \left(
\lambda _{1}-i\varpi \right) \Gamma \left( \lambda _{2}+i\varpi \right) 
\notag \\
&&\times \ _{\upsilon +1}F_{2\upsilon }\left[ -n,\Delta \left( \upsilon
,n+1-p\right) ;\Delta \left( \upsilon ,q+1\right) ,\Delta \left( \upsilon
,1-\gamma _{2}-i\varpi \right) ;1\right]  \notag \\
&&\times \frac{1}{\left( 1-\lambda _{1}+i\varpi \right) _{m}}\sum_{r=0}^{m}%
\frac{\left( a+b-m\right) _{r}\left( \lambda _{2}+i\varpi \right) _{r}}{%
\left( \lambda _{1}+\lambda _{2}-m\right) _{r}r!}\sum_{s=0}^{r}\frac{\left(
-r\right) _{s}}{s!}\left( \frac{s+b+1}{\upsilon }\right) _{m}.  \notag
\end{eqnarray}%
Choosing $\gamma _{1}+\gamma _{2}-1=q=b$ and $\lambda _{1}+\lambda
_{2}+1=p=a $ in relation (\ref{P1}), we obtain the biorthogonality relation (%
\ref{Mort}) in the left-hand side of the Parseval identity. So, we have%
\begin{eqnarray*}
&&\frac{n!\Gamma \left( \gamma _{2}+\lambda _{1}+1-n\right) \Gamma \left(
\gamma _{1}+\lambda _{2}+\upsilon n\right) }{\left( \gamma _{2}+\lambda
_{1}-n-\upsilon n\right) \Gamma \left( \gamma _{1}+\gamma _{2}+\lambda
_{1}+\lambda _{2}-n\right) }\delta _{n,m} \\
&=&\frac{\left( \gamma _{1}+\lambda _{2}\right) _{\upsilon n}\left(
1-\lambda _{1}-\lambda _{2}\right) _{n}}{2\pi i\Gamma \left( \gamma
_{1}+\gamma _{2}\right) \Gamma \left( \lambda _{1}+\lambda _{2}\right) }%
\int\limits_{-\infty }^{\infty }\Gamma \left( \gamma _{1}-i\varpi \right)
\Gamma \left( \gamma _{2}+i\varpi \right) \Gamma \left( \lambda _{1}-i\varpi
\right) \Gamma \left( \lambda _{2}+i\varpi \right) \\
&&\times \Phi \left( \gamma _{1},\gamma _{2},\lambda _{1},\lambda
_{2},n,\upsilon ;i\varpi \right) \chi \left( \gamma _{1},\gamma _{2},\lambda
_{1},\lambda _{2},m,\upsilon ;-i\varpi \right) d\varpi ,
\end{eqnarray*}%
where%
\begin{eqnarray*}
&&\Phi \left( \gamma _{1},\gamma _{2},\lambda _{1},\lambda _{2},n,\upsilon
;\varpi \right) \\
&=&\ _{\upsilon +1}F_{2\upsilon }\left[ -n,\Delta \left( \upsilon ,n-\gamma
_{2}-\lambda _{1}\right) ;\Delta \left( \upsilon ,\gamma _{1}+\lambda
_{2}\right) ,\Delta \left( \upsilon ,1-\gamma _{2}-\varpi \right) ;1\right]
\end{eqnarray*}%
and%
\begin{eqnarray*}
&&\chi \left( \gamma _{1},\gamma _{2},\lambda _{1},\lambda _{2},n,\upsilon
;\varpi \right) \\
&=&\frac{1}{\left( 1-\lambda _{1}-\varpi \right) _{n}}\sum_{r=0}^{n}\frac{%
\left( \gamma _{1}+\gamma _{2}+\lambda _{1}+\lambda _{2}-n\right) _{r}\left(
\lambda _{2}-\varpi \right) _{r}}{\left( \lambda _{1}+\lambda _{2}-n\right)
_{r}r!}\sum_{s=0}^{r}\frac{\left( -r\right) _{s}}{s!}\left( \frac{s+\gamma
_{1}+\lambda _{2}}{\upsilon }\right) _{n}.
\end{eqnarray*}%
Consequently, we can give the following theorem:

\begin{theorem}
The pair of functions%
\begin{eqnarray*}
&&\Phi \left( \gamma _{1},\gamma _{2},\lambda _{1},\lambda _{2},n,\upsilon
;x\right) \\
&=&\ _{\upsilon +1}F_{2\upsilon }\left[ -n,\Delta \left( \upsilon ,n-\gamma
_{2}-\lambda _{1}\right) ;\Delta \left( \upsilon ,\gamma _{1}+\lambda
_{2}\right) ,\Delta \left( \upsilon ,1-\gamma _{2}-x\right) ;1\right]
\end{eqnarray*}%
and%
\begin{eqnarray*}
&&\chi \left( \gamma _{1},\gamma _{2},\lambda _{1},\lambda _{2},n,\upsilon
;x\right) \\
&=&\frac{1}{\left( 1-\lambda _{1}-x\right) _{n}}\sum_{r=0}^{n}\frac{\left(
\gamma _{1}+\gamma _{2}+\lambda _{1}+\lambda _{2}-n\right) _{r}\left(
\lambda _{2}-x\right) _{r}}{\left( \lambda _{1}+\lambda _{2}-n\right) _{r}r!}%
\sum_{s=0}^{r}\frac{\left( -r\right) _{s}}{s!}\left( \frac{s+\gamma
_{1}+\lambda _{2}}{\upsilon }\right) _{n}.
\end{eqnarray*}%
are finite biorthogonal w.r.t. the weight function $w\left( x\right) =\Gamma
\left( \gamma _{1}-x\right) \Gamma \left( \gamma _{2}+x\right) \Gamma \left(
\lambda _{1}-x\right) \Gamma \left( \lambda _{2}+x\right) $ on $\left(
-\infty ,\infty \right) $, and they have the following finite
biorthogonality relation%
\begin{eqnarray*}
&&\int\limits_{-\infty }^{\infty }\Gamma \left( \gamma _{1}-ix\right) \Gamma
\left( \gamma _{2}+ix\right) \Gamma \left( \lambda _{1}-ix\right) \Gamma
\left( \lambda _{2}+ix\right) \\
&&\times \Phi \left( \gamma _{1},\gamma _{2},\lambda _{1},\lambda
_{2},n,\upsilon ;ix\right) \chi \left( \gamma _{1},\gamma _{2},\lambda
_{1},\lambda _{2},m,\upsilon ;-ix\right) dx \\
&=&n!2\pi iB\left( \gamma _{1}+\gamma _{2},\lambda _{1}+\lambda _{2}\right) 
\frac{\Gamma \left( \gamma _{1}+\lambda _{2}\right) \Gamma \left( \gamma
_{2}+\lambda _{1}+1\right) \left( 1-\gamma _{1}-\gamma _{2}-\lambda
_{1}-\lambda _{2}\right) _{n}}{\left( \gamma _{2}+\lambda _{1}-n-\upsilon
n\right) \left( 1-\lambda _{1}-\lambda _{2}\right) _{n}\left( -\gamma
_{2}-\lambda _{1}\right) _{n}}\delta _{n,m}
\end{eqnarray*}%
for $\gamma _{1},\gamma _{2},\lambda _{1},\lambda _{2}>0$ and $\gamma
_{2}+\lambda _{1}>n+\upsilon n$.
\end{theorem}

\section{Integral and operational representations for $M_{n}\left( p,q,%
\protect\upsilon ;x\right) $}

In this section, we give the operational and integral representations for
the finite biorthogonal polynomials suggested by the finite univariate
orthogonal polynomials $M_{n}^{\left( p,q\right) }\left( x\right) $.

We recall that%
\begin{equation*}
D_{x}^{-1}f\left( x\right) =\int\limits_{0}^{x}f\left( t\right) dt.
\end{equation*}

\begin{theorem}
The finite univariate orthogonal polynomials $M_{n}\left( p,q,\upsilon
;x\right) $ have the following operational representation 
\begin{equation*}
M_{n}\left( p,q,\upsilon ;x\right) =\frac{\left( -1\right) ^{n}\left(
q+1\right) _{\upsilon n}}{x^{q}}\ _{\upsilon +1}F_{0}\left[ \QATOP{-n,\Delta
\left( \upsilon ,n+1-p\right) ;}{-;}\left( \frac{-1}{D_{x}}\right)
^{\upsilon }\right] \left\{ \frac{x^{q}}{\Gamma \left( q+1\right) }\right\} ,
\end{equation*}%
where$\ _{p}F_{q}$ is the generalized hypergeometric functions.
\end{theorem}

\begin{proof}
Polynomials $M_{n}\left( p,q,\upsilon ;x\right) $ can be written of form%
\begin{eqnarray*}
M_{n}\left( p,q,\upsilon ;x\right) &=&\frac{\left( -1\right) ^{n}\left(
q+1\right) _{\upsilon n}}{x^{q}}\sum\limits_{j=0}^{n}\frac{\left( -n\right)
_{j}\left( n+1-p\right) _{\upsilon j}}{j!}\left( -D_{x}^{-1}\right)
^{\upsilon j}\left\{ \frac{x^{q}}{\Gamma \left( q+1\right) }\right\} \\
&=&\frac{\left( -1\right) ^{n}\left( q+1\right) _{\upsilon n}}{x^{q}}%
\sum\limits_{j=0}^{n}\frac{\left( -n\right) _{j}\left( n+1-p\right)
_{\upsilon j}}{j!}\left( \frac{-1}{D_{x}}\right) ^{\upsilon j}\left\{ \frac{%
x^{q}}{\Gamma \left( q+1\right) }\right\} \\
&=&\frac{\left( -1\right) ^{n}\left( q+1\right) _{\upsilon n}}{x^{q}}\
_{\upsilon +1}F_{0}\left[ \QATOP{-n,\Delta \left( \upsilon ,n+1-p\right) ;}{%
-;}\left( \frac{-1}{D_{x}}\right) ^{\upsilon }\right] \left\{ \frac{x^{q}}{%
\Gamma \left( q+1\right) }\right\} .
\end{eqnarray*}
\end{proof}

\bigskip

Now, let's recall the following representations of the Gamma function and
the incomplete Gamma function, respectively:%
\begin{equation}
\Gamma \left( z\right) =\int\limits_{0}^{\infty }e^{-t}t^{z-1}dt\ \ \ \ \
\left( \func{Re}\left( z\right) >0\right)  \label{Gamma}
\end{equation}%
and%
\begin{equation}
\frac{1}{\Gamma \left( z\right) }=\frac{1}{2\pi i}\int\limits_{-\infty
}^{0^{+}}e^{\theta }\theta ^{-z}d\theta \ \ \ \ \ \left( \left\vert \arg
\left( \theta \right) \right\vert \leq \pi \right) .  \label{PGamma}
\end{equation}%
The relation (\ref{PGamma}) is known as the Hankel's representation \cite{WW}%
, where the contour of integration (\ref{PGamma}) is the standard Hankel
contour.

\begin{theorem}
The finite univariate orthogonal polynomials $M_{n}\left( p,q,\upsilon
;x\right) $ have the following integral representation:%
\begin{equation*}
M_{n}\left( p,q,\upsilon ;x\right) =\frac{\Gamma \left( \upsilon
n+q+1\right) }{2\pi i\Gamma \left( n-p+1\right) }\int\limits_{0}^{\infty
}\int\limits_{-\infty }^{0^{+}}e^{-u+v}u^{n-p}v^{-\left( \upsilon
n+q+1\right) }\left( \left( -ux\right) ^{\upsilon }-v^{\upsilon }\right)
^{n}dvdu.
\end{equation*}
\end{theorem}

\begin{proof}
Considering (\ref{Gamma}) and (\ref{PGamma}) we have%
\begin{eqnarray*}
M_{n}\left( p,q,\upsilon ;x\right) &=&\frac{\left( -1\right) ^{n}\Gamma
\left( \upsilon n+q+1\right) }{2\pi i\Gamma \left( n-p+1\right) }%
\int\limits_{0}^{\infty }\int\limits_{-\infty
}^{0^{+}}e^{-u+v}u^{n-p}v^{-q-1}\sum\limits_{j=0}^{n}\frac{\left( -n\right)
_{j}}{j!}\left( \left( -\frac{ux}{v}\right) ^{\upsilon }\right) ^{j}dudv \\
&=&\frac{\Gamma \left( \upsilon n+q+1\right) }{2\pi i\Gamma \left(
n-p+1\right) }\int\limits_{0}^{\infty }\int\limits_{-\infty
}^{0^{+}}e^{-u+v}u^{n-p}v^{-q-1}\left( \left( -\frac{ux}{v}\right)
^{\upsilon }-1\right) ^{n}dudv \\
&=&\frac{\Gamma \left( \upsilon n+q+1\right) }{2\pi i\Gamma \left(
n-p+1\right) }\int\limits_{0}^{\infty }\int\limits_{-\infty
}^{0^{+}}e^{-u+v}u^{n-p}v^{-\left( \upsilon n+q+1\right) }\left( \left(
-ux\right) ^{\upsilon }-v^{\upsilon }\right) ^{n}dudv.
\end{eqnarray*}
\end{proof}

\section{Laplace transform and Fractional calculus operators for finite
biorthogonal polynomials $M_{n}\left( p,q,\protect\upsilon ;x\right) $}

In this section we calculate the Laplace transform of $M_{n}\left(
p,q,\upsilon ;x\right) $. Then, we compute fractional calculus operators.

The Laplace transform of the function $f$ is defined by \cite{KG}%
\begin{equation*}
\tciLaplace \lbrack f]\left( s\right) =\int\limits_{0}^{\infty
}e^{-st}f\left( t\right) dt\ \ \ \ \left( \func{Re}\left( s\right) >0\right)
.
\end{equation*}

\begin{theorem}
The Laplace transform of $M_{n}\left( p,q,\upsilon ;x\right) $ is given by%
\begin{equation*}
\tciLaplace \left\{ x^{q}M_{n}\left( p,q,\upsilon ;wx\right) \right\} =\frac{%
\left( -1\right) ^{n}\Gamma \left( \upsilon n+q+1\right) }{a^{q+1}}\
_{\upsilon +1}F_{0}\left[ \QATOP{-n,\Delta \left( \upsilon ,n+1-p\right) ;}{%
-;}\left( \frac{-w}{a}\right) ^{\upsilon }\right] ,
\end{equation*}%
for $\left\vert \left( \frac{w}{a}\right) ^{\upsilon }\right\vert <1$ and $%
q>-1$.
\end{theorem}

\begin{proof}
\begin{eqnarray*}
L\left\{ x^{q}M_{n}\left( p,q,\upsilon ;wx\right) \right\}
&=&\int\limits_{0}^{\infty }e^{-\alpha x}x^{q}\left( -1\right) ^{n}\left(
q+1\right) _{\upsilon n}\sum\limits_{j=0}^{n}\left( -1\right) ^{j}\binom{n}{j%
}\frac{\left( n+1-p\right) _{\upsilon j}}{\left( q+1\right) _{\upsilon j}}%
\left( -wx\right) ^{\upsilon j}dx \\
&=&\frac{1}{\alpha ^{q+1}}\left( -1\right) ^{n}\Gamma \left( \upsilon
n+q+1\right) \sum\limits_{j=0}^{n}\left( -1\right) ^{j}\binom{n}{j}\left(
n+1-p\right) _{\upsilon j}\left( \frac{-w}{\alpha }\right) ^{\upsilon j} \\
&=&\frac{\left( -1\right) ^{n}\Gamma \left( \upsilon n+q+1\right) }{\alpha
^{q+1}}\ _{\upsilon +1}F_{0}\left[ -n,\Delta \left( \upsilon ,n+1-p\right)
;-;\left( \frac{-w}{\alpha }\right) ^{\upsilon }\right] .
\end{eqnarray*}
\end{proof}

\bigskip

The Riemann-Liouville fractional integral is defined by \cite{KST}%
\begin{equation*}
_{x}I_{a^{+}}^{\mu }\left( f\right) =\frac{1}{\Gamma \left( \mu \right) }%
\int\limits_{a}^{x}\left( x-t\right) ^{\mu -1}f\left( t\right) dt,\ \ \ \
f\in L^{1}\left[ a,b\right] ,
\end{equation*}%
where $\mu \in 
\mathbb{C}
$,$\ \func{Re}\left( \mu \right) >0$,$\ x>a$.

The Riemann-Liouville fractional derivativen is defined by \cite{KST}%
\begin{equation*}
_{x}D_{a^{+}}^{\lambda }\left( f\right) =\left( \frac{d}{dx}\right) ^{n}\
_{x}I_{a^{+}}^{n-\lambda }\left( f\right) ,\ \ \ \ f\in C^{n}\left[ a,b%
\right] ,
\end{equation*}%
where $\lambda \in 
\mathbb{C}
$, $\func{Re}\left( \lambda \right) >0$, $\left[ \func{Re}\left( \lambda
\right) \right] $ is the integral part of $\func{Re}\left( \lambda \right) $%
, $n=\left[ \func{Re}\left( \lambda \right) \right] +1$ and $x>a$.

\begin{theorem}
The set of the polynomials $M_{n}\left( p,q,\upsilon ;x\right) $ have the
following Riemann-Liouville fractional integral operator%
\begin{equation*}
\ _{x}I_{a^{+}}^{\mu }\left[ \left( x-a\right) ^{q}M_{n}\left( p,q,\upsilon
;w\left( x-a\right) \right) \right] =\frac{\left( x-a\right) ^{q+\mu }\Gamma
\left( \upsilon n+q+1\right) }{\Gamma \left( \upsilon n+q+\mu +1\right) }%
M_{n}\left( p,q+\mu ,\upsilon ;w\left( x-a\right) \right) .
\end{equation*}
\end{theorem}

\begin{proof}
\begin{eqnarray*}
&&_{x}I_{a^{+}}^{\mu }\left[ \left( x-a\right) ^{q}M_{n}\left( p,q,\upsilon
;w\left( x-a\right) \right) \right] \\
&=&\frac{1}{\Gamma \left( \mu \right) }\left( -1\right) ^{n}\left(
q+1\right) _{\upsilon n}\sum\limits_{j=0}^{n}\left( -1\right) ^{j}\binom{n}{j%
}\frac{\left( n+1-p\right) _{\upsilon j}}{\left( q+1\right) _{\upsilon j}}%
\left( -w\right) ^{\upsilon j}\int\limits_{a}^{x}\left( x-t\right) ^{\mu
-1}\left( t-a\right) ^{q+\upsilon j}dt \\
&=&\frac{1}{\Gamma \left( \mu \right) }\left( -1\right) ^{n}\left(
q+1\right) _{\upsilon n}\sum\limits_{j=0}^{n}\left( -1\right) ^{j}\binom{n}{j%
}\frac{\left( n+1-p\right) _{\upsilon j}}{\left( q+1\right) _{\upsilon j}}%
\left( -w\left( x-a\right) \right) ^{\upsilon j}\int\limits_{0}^{1}\left(
1-u\right) ^{\mu -1}u^{q+\upsilon j}du \\
&=&\left( x-a\right) ^{q+\mu }\left( -1\right) ^{n}\left( q+1\right)
_{\upsilon n}\sum\limits_{j=0}^{n}\left( -1\right) ^{j}\binom{n}{j}\frac{%
\Gamma \left( q+1+\upsilon j\right) \left( n+1-p\right) _{\upsilon j}\left(
-w\left( x-a\right) \right) ^{\upsilon j}}{\Gamma \left( q+\mu +1+\upsilon
j\right) \left( q+1\right) _{\upsilon j}} \\
&=&\left( x-a\right) ^{q+\mu }\frac{\Gamma \left( q+1+\upsilon n\right) }{%
\Gamma \left( q+\mu +1+\upsilon n\right) }M_{n}\left( p,q+\mu ,\upsilon
;w\left( x-a\right) \right) .
\end{eqnarray*}
\end{proof}

\begin{theorem}
For polynomials $M_{n}\left( p,q,\upsilon ;x\right) $, the Riemann-Liouville
fractional derivative operator%
\begin{equation*}
\ _{x}D_{a^{+}}^{\lambda }\left[ \left( x-a\right) ^{q}M_{n}\left(
p,q,\upsilon ;w\left( x-a\right) \right) \right] =\frac{\left( x-a\right)
^{q-\lambda }\Gamma \left( \upsilon n+q+1\right) }{\Gamma \left( \upsilon
n+q-\lambda +1\right) }M_{n}\left( p,q-\lambda ,\upsilon ;w\left( x-a\right)
\right)
\end{equation*}%
holds.
\end{theorem}

\begin{proof}
\begin{eqnarray*}
&&\ _{x}D_{a^{+}}^{\lambda }\left[ \left( x-a\right) ^{q}M_{n}\left(
p,q,\upsilon ;w\left( x-a\right) \right) \right] =D_{x}^{r}\
_{x}I_{a^{+}}^{r-\lambda }\left[ \left( x-a\right) ^{q}M_{n}\left(
p,q,\upsilon ;w\left( x-a\right) \right) \right] \\
&=&\frac{1}{\Gamma \left( r-\lambda \right) }\left( -1\right) ^{n}\left(
q+1\right) _{\upsilon n}\sum\limits_{j=0}^{n}\left( -1\right) ^{j}\binom{n}{j%
}\frac{\left( n+1-p\right) _{\upsilon j}}{\left( q+1\right) _{\upsilon j}}%
\left( -w\right) ^{\upsilon j} \\
&&\times D_{x}^{r}\int\limits_{a}^{r}\left( x-\zeta \right) ^{r-\lambda
-1}\left( \zeta -a\right) ^{q+\upsilon j}d\zeta \\
&=&\left( -1\right) ^{n}\Gamma \left( \upsilon n+q+1\right)
\sum\limits_{j=0}^{n}\left( -1\right) ^{j}\binom{n}{j}\frac{\left(
n+1-p\right) _{\upsilon j}}{\Gamma \left( q-\lambda +\upsilon j+1\right) }%
\left( -w\right) ^{\upsilon j}\left( x-a\right) ^{q-\lambda +\upsilon j} \\
&=&\left( x-a\right) ^{q-\lambda }\frac{\Gamma \left( \upsilon n+q+1\right) 
}{\Gamma \left( \upsilon n+q-\lambda +1\right) }\left[ \left( -1\right)
^{n}\left( q-\lambda +1\right) _{\upsilon n}\right. \\
&&\left. \times \sum\limits_{j=0}^{n}\left( -1\right) ^{j}\binom{n}{j}\frac{%
\left( n+1-p\right) _{\upsilon j}}{\left( q-\lambda +1\right) _{\upsilon j}}%
\left( -w\left( x-a\right) \right) ^{\upsilon j}\right] \\
&=&\left( x-a\right) ^{q-\lambda }\frac{\Gamma \left( q+1+\upsilon n\right) 
}{\Gamma \left( q-\lambda +1+\upsilon n\right) }M_{n}\left( p,q-\lambda
,\upsilon ;w\left( x-a\right) \right) .
\end{eqnarray*}
\end{proof}

\section{Concluding remarks}

In this study, a pair of finite univariate biorthogonal polynomials
suggested by finite orthogonal polynomials $M_{n}^{\left( p,q\right) }\left(
x\right) $ is presented for the first time in the literature. We give some
important properties of them. A purpose of this study is to obtain
somefamilies of finite biorthogonal functions using Parseval identity.
Another aim of this paper is to construct a new fractional calculus based on
the finite biorthogonal polynomials suggested by finite orthogonal
polynomials $M_{n}^{\left( p,q\right) }\left( x\right) $. Also, we introduce
some useful limit and connection relations.

It can be noted that, in the special case of $p=n+1$ and $p=n-q$, the
polynomials $M_{n}\left( p,q,\upsilon ;x\right) $ and $\mathfrak{M}%
_{n}\left( p,q,\upsilon ;x\right) $ defined by (\ref{9}) and (\ref{11}) are
related to the first and second sets of the Konhauser polynomials, $%
Z_{n}^{\left( \gamma \right) }(x;\upsilon )$ and $Y_{n}^{\left( \gamma
\right) }(x;\upsilon )$, respectively:%
\begin{equation*}
M_{n}\left( n+1,q,\upsilon ;x\right) =\left( -1\right) ^{n}n!Z_{n}^{\left(
q\right) }(-x;\upsilon )
\end{equation*}%
and%
\begin{equation*}
\mathfrak{M}_{n}\left( n-q,q,\upsilon ;x\right) =\left( -\left( 1+x\right)
\right) ^{n}n!Y_{n}^{\left( q\right) }(\frac{x}{1+x};\upsilon ).
\end{equation*}

Finite biorthogonal polynomial sets are a new research interest. This new
subject enables researchers to investigate many properties and applications
of the finite biorthogonal set. Moreover, many extensions like q-analogs and
multivariate forms of this class, of which very few properties are known
yet, can be studied.

\begin{center}
$\mathtt{FUNDING}$
\end{center}

The author G\"{u}ldo\u{g}an Lekesiz E. in the work has been partially
supported by the Scientific and Technological Research Council of Turkey
(TUBITAK) (Grant number 2218-122C240).\bigskip

\begin{center}
DECLARATIONS
\end{center}

$\mathbf{Conflict\ of\ Interest\ }$The authors declared that they have no
conflict of interest.

\end{document}